\documentclass[12pt]{amsart}
\usepackage{cleveref,mathrsfs,url}

\title[On the PI-exp of algebras with generalized actions]{On the PI-exponent of matrix algebras and algebras with generalized actions}
\author{Thiago Castilho de Mello}
\address{Instituto de Ciência e Tecnologia, Universidade Federal de S\~ao Paulo, SP, Brazil}
\email{tcmello@unifesp.br}
\author{Felipe Yukihide Yasumura}
\address{Department of Mathematics, Instituto de Matem\'atica e Estat\'istica, Universidade de S\~ao Paulo, SP, Brazil}
\email{fyyasumura@ime.usp.br}

\thanks{This work is supported by S\~ao Paulo Research Foundation (FAPESP), grant 2025/02457-5, and by 
CNPq, grant 405779/2023-2.}
\dedicatory{Dedicated to Professor C\'esar Polcino Milies on his 80th birthday.}

\newtheorem{Thm}{Theorem}
\newtheorem{Lemma}[Thm]{Lemma}
\newtheorem{proposition}[Thm]{Proposition}
\newtheorem{Cor}[Thm]{Corollary}
\theoremstyle{definition}

\theoremstyle{remark}
\newtheorem*{Remark}{Remark}
\newtheorem{Example}{Example}

\begin{document}

\begin{abstract}
We compute the PI-exponent of the matrix ring with coefficients in an associative algebra. As a consequence, we prove the following. Let $\mathcal{R}$ be a PI-algebra with a positive PI-exponent. If $M_n(\mathcal{R})$ and $M_m(\mathcal{R})$ satisfy the same set of polynomial identities then $n=m$. We provide examples where this result fails if either $\mathcal{R}$ is not PI or has zero exponent. We obtain the same statement for certain finite-dimensional algebras with generalized action over an algebraically closed field of zero characteristic.
\end{abstract}
\maketitle

\section{Introduction}
We aim to answer the question, posed in \cite{deMello}: given an associative algebra $\mathcal{R}$, if $M_n(\mathcal {R})$ and $M_m(\mathcal{R})$ satisfy the same polynomial identities is it true that $m = n$? To present a positive answer, we rely on the groundbreaking theory developed by Giambruno and Zaicev, on the exponent of a PI-algebra.

Let $\mathcal{A}$ be an associative PI-algebra over an algebraically closed field of characteristic zero.  By $\mathrm{Id}(\mathcal{A})$ we denote the set of all polynomial identities of $\mathcal A$ in the free associative algebra $\mathbb{F}\langle X \rangle$. We define the codimension sequence of $\mathcal{A}$ via
$$
c_m(\mathcal{A})=\dim\left(P_m/P_m\cap\mathrm{Id}(\mathcal{A})\right),\quad m\in\mathbb{N}.
$$
The \emph{PI-exponent} of $\mathcal{A}$ is
$$
\mathrm{exp}(\mathcal{A})=\lim\limits_{m\to\infty}\sqrt[m]{c_m(\mathcal{A})}.
$$
A celebrated result by Giambruno and Zaicev in the early 2000 (see \cite{GZ1999}, and also the monograph \cite{GZ2005}) computes the PI-exponent of the Grassmann envelope of a finite-dimensional associative $\mathbb{Z}_2$-graded algebra. More precisely, let $\mathcal{B}$ be a finite-dimensional associative $\mathbb{Z}_2$-graded algebra. The \emph{Grassmann envelope} of $\mathcal{B}$ is $G(\mathcal{B})=\mathcal{B}_0\otimes G_0\oplus\mathcal{B}_1\otimes G_1$, where $G=G_0\oplus G_1$ is the infinite-dimensional Grassmann algebra endowed with the natural $\mathbb{Z}_2$-grading. Let $\mathcal{B}=\mathcal{S}+J(\mathcal{B})$ be the Wedderburn-Malcev decomposition of $\mathcal{B}$, that is, $J(\mathcal{B})$ is its Jacobson radical, and $\mathcal{S}$ is a semisimple $\mathbb{Z}_2$-graded algebra such that $\mathcal{S}\cap J(\mathcal{B})=0$. Write $\mathcal{S}=\mathcal{B}_1\oplus\cdots\oplus\mathcal{B}_r$, where each $\mathcal{B}_i$ is a graded-simple $\mathbb{Z}_2$-graded algebra. A sequence of distinct integers between $1$ and $r$, say $I=(i_1,\ldots,i_s)$, is \emph{admissible} if $\mathcal{B}_{i_1}J(\mathcal{B})\mathcal{B}_{i_2}J(\mathcal{B})\cdots J(\mathcal{B})\mathcal{B}_{i_s}\ne0$. Then (see \cite[Proposition 6.5.1]{GZ2005})
$$
\mathrm{exp}(G(\mathcal{B}))=\mathrm{max}\left\{\dim(\mathcal{B}_{i_1}\oplus\cdots\oplus\mathcal{B}_{i_s})\mid I=(i_1,\ldots,i_s)\text{ is admissible}\right\}.
$$
Finally, Kemer's theory says that, for each associative PI-algebra $\mathcal{A}$, there exists a finite-dimensional $\mathbb{Z}_2$-graded algebra $\mathcal{B}$ such that $\mathrm{Id}(\mathcal{A})=\mathrm{Id}(G(\mathcal{B}))$ (see \cite{kemerbook}, and see also \cite{AGPR}). As a consequence, the PI-exponent of any associative algebra over a field of characteristic zero, exists and is an integer.

It is worth mentioning that the above result is generalized to several contexts. For instance, we mention \cite{Gord,Gordb} where the same kind of result is obtained in the context of algebras with generalized action (see the definition below), which includes for instance graded algebras and algebras with an involution.

In this paper, we prove that $\mathrm{exp}(M_n(\mathcal{R}))=n^2\mathrm{exp}(\mathcal{R})$, for an associative PI-algebra $\mathcal{R}$. As a consequence, we have a positive answer for the question posed in \cite{deMello} if $\mathcal{R}$ is a PI-algebra with positive exponent. In addition, we also prove that $\mathrm{exp}^{H\otimes K}(\mathcal{A}\otimes\mathcal{S})=\mathrm{exp}^H(\mathcal{A})\cdot\dim\mathcal{S}$, where $\mathcal{A}$ is a finite-dimensional algebra endowed with a generalized $H$-action, under some addition conditions, and $\mathcal{S}$ is a unital central $K$-simple finite-dimensional algebra.

\section{Main result}
Every vector space and algebra will be over a fixed field of characteristic zero $\mathbb{F}$.

The following known result will be necessary:
\begin{Lemma}[{\cite[Theorem 7.2.3]{AGPR}}]\label{regev}
Let $\mathcal{A}_i$, $\mathcal{B}_i$ be $\mathbb{F}$-algebras satisfying $\mathrm{Id}(\mathcal{A}_i)=\mathrm{Id}(\mathcal{B}_i)$, for $i=1,2$. Then $\mathrm{Id}(\mathcal{A}_1\otimes\mathcal{A}_2)=\mathrm{Id}(\mathcal{B}_1\otimes\mathcal{B}_2)$. In particular, for each $n\in\mathbb{N}$, $\mathrm{Id}(M_n(\mathcal{A}_1))=\mathrm{Id}(M_n(\mathcal{B}_1))$.
\end{Lemma}

Our main result is as follows:
\begin{Thm}\label{mainthm}
Let $\mathcal{A}$ be an associative algebra and $n\in\mathbb{N}$. Then
$$
\mathrm{exp}(M_n(\mathcal{A}))=n^2\mathrm{exp}(\mathcal{A}).
$$
\end{Thm}
\begin{proof}
If $\mathcal{A}$ is not a PI-algebra, then $M_n(\mathcal{A})$ is not PI as well since it contains a copy of $\mathcal{A}$ via $a\mapsto\mathrm{diag}(a,a,\ldots,a)$. Thus, in this case, we obtain the equality of the exponents. Hence, we may assume that $\mathcal{A}$ is a PI-algebra.

Let $\mathcal{B}$ be a finite-dimensional $\mathbb{Z}_2$-graded algebra such that $\mathrm{Id}(\mathcal{A})=\mathrm{Id}(G(\mathcal{B}))$. Let $\mathcal{B}=\mathcal{B}^0\oplus\mathcal{B}^1$ be its decomposition in homogeneous subspaces. The algebra $M_n(\mathcal{B})$ is $\mathbb{Z}_2$-graded via
$$
M_n(\mathcal{B})^i=M_n(\mathcal{B}^i).
$$
Note that $G(M_n(\mathcal{B}))=M_n(G(\mathcal{B}))$. Moreover, from \Cref{regev}, one has
$$
\mathrm{Id}(M_n(\mathcal{A}))=\mathrm{Id}(M_n(G(\mathcal{B})))=\mathrm{Id}(G(M_n(\mathcal{B}))).
$$
Let $\mathcal{B}=\mathcal{S}+J(\mathcal{B})$ be its Wedderburn-Malcev decomposition, where $\mathcal{S}$ is $\mathbb{Z}_2$-graded semisimple. Since $\mathcal{B}$ is finite-dimensional, its Jacobson radical is a nilpotent ideal. Thus, $J:=M_n(J(\mathcal{B}))$ is a nilpotent graded ideal. Moreover,
$$
M_n(\mathcal{B})/J\cong M_n(\mathcal{S})\cong M_n(\mathbb{F})\otimes\mathcal{S}
$$
is semisimple. Hence, $J$ is the Jacobson radical of $M_n(\mathcal{B})$. Therefore,
$$
M_n(\mathcal{B})=M_n(\mathcal{S})+J
$$
is the Wedderburn-Malcev decomposition of $M_n(\mathcal{B})$. Let $$
\mathcal{S}=\mathcal{B}_1\oplus\cdots\oplus\mathcal{B}_t
$$
be the decomposition of $\mathcal{S}$ as a direct sum of $\mathbb{Z}_2$-graded-simple algebras. Let $\mathcal{B}_{i_1}\oplus\cdots\oplus\mathcal{B}_{i_r}\subseteq\mathcal{S}$ be such that $\mathcal{B}_{i_1}J(\mathcal{B})\cdots J(\mathcal{B})\mathcal{B}_{i_r}\ne0$. It means that we can find $b_j\in\mathcal{B}_{i_j}$ and $u_1$, \dots, $u_{r-1}\in J(\mathcal{B})$ such that $u'=b_1u_1\cdots u_{r-1}b_r\ne0$. Notice that $M_n(\mathcal{B}_{i_1})\oplus\cdots\oplus M_n(\mathcal{B}_{i_r})\subseteq M_n(\mathcal{S})$. Now,
$$
(b_1 e_{11})(u_1 e_{11})\cdots(u_{r-1}e_{11})(b_re_{11})=u e_{11}\ne0.
$$
Hence, $M_n(\mathcal{B}_{i_1})J\cdots JM_n(\mathcal{B}_{i_r})\ne0$. Thus,
$$
\mathrm{exp}(G(M_n(\mathcal{B})))\ge n^2\mathrm{exp}(G(\mathcal{B})). 
$$
On the other hand, from \cite[Theorem 4.2.5]{GZ2005}, one has
$$
c_m(M_n(G(\mathcal{B}))\le c_m(M_n(\mathbb{F}))c_m(G(\mathcal{B})),
$$
for all $n\in\mathbb{N}$. It means,
$$
\mathrm{exp}(M_n(G(\mathcal{B})))\le\mathrm{exp}(M_n(\mathbb{F}))\mathrm{exp}(G(\mathcal{B}))=n^2\mathrm{exp}(G(\mathcal{B})).
$$

Combining both inequalities, we obtain that
$$
\mathrm{exp}(M_n(\mathcal{A}))=\mathrm{exp}(G(M_n(\mathcal{B})))=n^2\mathrm{exp}(G(\mathcal{B}))=n^2\mathrm{exp}(\mathcal{A}).
$$
The proof is complete.
\end{proof}

In general, it is not true that the exponent of a tensor product of two PI-algebras is the product of their respective exponents.
\begin{Example}
Let $\mathcal{A}=\mathrm{UT}_2(\mathbb{F})$. We know that $\mathrm{exp}(\mathcal{A})=2$. Furthermore, $\mathcal{A}\otimes\mathcal{A}\cong I(X)$, where $I(X)$ is the incidence algebra of the poset $X=\{1,2,3,4\}$, where $1\le2,3\le 4$. In other words,
$$
\mathrm{UT}_2\otimes\mathrm{UT}_2=\left\{\left(\begin{array}{cccc}\ast&\ast&\ast&\ast\\&\ast&0&\ast\\&&\ast&\ast\\&&&\ast\end{array}\right)\right\}.
$$
On one hand, $J(\mathrm{UT}_2\otimes\mathrm{UT}_2)^3=0$, so $\mathrm{exp}(\mathrm{UT}_2\otimes\mathrm{UT}_2)\le3$. On the other hand, the semisimple part of $\mathrm{UT}_2\otimes\mathrm{UT}_2$ is the set of diagonal matrices, which is isomorphic to a sum of four copies of the base field. Moreover, $e_{11}e_{12}e_{22}e_{24}e_{44}\ne0$, so the sequence $I=(1,2,4)$ is admissible. Hence, $\mathrm{exp}(\mathrm{UT}_2\otimes\mathrm{UT}_2)=3$.
\end{Example}

As a consequence of our result, we obtain:

\begin{Cor}\label{consequence}
Let $\mathcal{R}$ be a PI-algebra with $\mathrm{exp}(\mathcal{R})\ge1$ (for instance, $\mathcal{R}$ is unital). Then $\mathrm{Id}(M_r(\mathcal{R}))\subseteq\mathrm{Id}(M_s(\mathcal{R}))$ if and only if $r\ge s$. In particular, $\mathrm{Id}(M_r(\mathcal{R}))=\mathrm{Id}(M_s(\mathcal{R}))$ if and only if $r=s$.
\end{Cor}
\begin{proof}
If $r\le s$, then we can embed $M_r(\mathcal{R})$ in $M_s(\mathcal{R})$ via
$$
A\mapsto\left(\begin{array}{cc}A&0\\0&0\end{array}\right).
$$
Hence, $\mathrm{Id}(M_s(\mathcal{R}))\subseteq\mathrm{Id}(M_r(\mathcal{R}))$. On the other hand, if $\mathrm{Id}(M_s(\mathcal{R}))\subseteq\mathrm{Id}(M_r(\mathcal{R}))$, then we have $\mathrm{exp}(M_r(\mathcal{R}))\le\mathrm{exp}(M_s(\mathcal{R}))$. From the previous theorem, one has $r^2\mathrm{exp}(\mathcal{R})\le s^s\mathrm{exp}(\mathcal{R})$. Since $\mathrm{exp}(\mathcal{R})\ne0$, one obtains $r\le s$.
\end{proof}

The next two examples show that the conditions posed in the statement are essential.
\begin{Example}
The condition $\mathrm{exp}(\mathcal{R})>0$ in \Cref{consequence} is essential. Indeed, let $\mathcal{R}$ be an infinite-dimensional vector space and define the zero-product on $\mathcal{R}$, that is, $\mathcal{R}^2=0$. Since $\mathcal{R}$ is nilpotent, clearly $\mathrm{exp}(\mathcal{R})=0$. Then, $M_n(\mathcal{R})$ is an infinite-dimensional algebra where $M_n(\mathcal{R})^2=0$. An algebra isomorphism between two algebras endowed with a zero-product is just a vector space isomorphism. The dimension of $M_n(\mathcal{R})$ is the same dimension of $\mathcal{R}$. Hence, for any $m$, $n\in\mathbb{N}$, one has $M_n(\mathcal{R})\cong\mathcal{R}\cong M_m(\mathcal{R})$.
\end{Example}

\begin{Example}
The condition of $\mathcal{R}$ being a PI-algebra in \Cref{consequence} is also essential. For instance, let $\mathcal{V}$ be a infinite-dimensional vector space with a countable basis, and let $\mathcal{R}=\mathcal{V}\otimes\mathcal{V}^\ast\subseteq\mathrm{End}_\mathbb{F}(\mathcal{V})$ be the subalgebra of finite-rank linear operators. Then, for any $m$, $n\in\mathbb{N}$, one has
$$
M_n(\mathcal{R})\cong\mathcal{R}\cong M_m(\mathcal{R}).
$$
\end{Example}

\begin{Remark}
Let $\mathcal{R}$ be a unital PI-algebra. Notice that $\mathcal{R}$ has positive PI-exponent. Then, \Cref{consequence} tells us that $M_n(\mathcal{R})\cong M_n(\mathcal{R})$ implies $m=n$.
\end{Remark}
In view of \Cref{consequence}, we obtain:
\begin{Cor}
Let $\mathcal{R}$ be a unital PI-algebra. Then $\mathcal{R}$ is an IBN (invariant basis number) ring.
\end{Cor}
\begin{proof}
Assume that $\mathcal{R}^m\cong\mathcal{R}^n$ as $\mathcal{R}$-modules, for some $m$, $n\in\mathbb{N}$. Recall that an isomorphism of $\mathcal{R}$-modules $\mathcal{V}\to\mathcal{W}$ implies an algebra isomorphism $\mathrm{End}_\mathcal{R}(\mathcal{V})\to\mathrm{End}_\mathcal{R}(\mathcal{W})$. In addition, since $\mathcal{R}$ is unital, we obtain an algebra isomorphism $\mathrm{End}_\mathcal{R}(\mathcal{R}^n)\cong M_n(\mathcal{R})$. Thus, 
$$
M_m(\mathcal{R})\cong\mathrm{End}_\mathcal{R}(\mathcal{R}^m)\cong\mathrm{End}_\mathcal{R}(\mathcal{R}^n)\cong M_n(\mathcal{R}).
$$
From \Cref{consequence}, we get $n=m$. Hence, $\mathcal{R}$ is an IBN ring.
\end{proof}

\begin{Remark}
The previous result is known: let $\mathcal{R}$ be a unital PI-ring. Then, there exists a maximal ideal $\mathcal{I}\subseteq\mathcal{R}$. So, $\mathcal{R}/\mathcal{I}$ is a simple PI-ring. By Kaplansky's Theorem, every simple PI-ring is a simple Artinian ring. Thus, $\mathcal{R}/\mathcal{I}$ is an IBN ring. Since we have a unital homomorphism $\mathcal{R}\to\mathcal{R}/\mathcal{I}$, we get that $\mathcal{R}$ is an IBN ring (see, for instance, \cite[Remark 1.5]{Lam}).
\end{Remark}

\section{A generalization of the main result}
In this section, we extend the construction of the previous section to the context of finite-dimensional associative algebras endowed with a generalized $H$-action.

\subsection{Generalized $H$-action}
Let $H$ be an associative algebra. We say that an associative algebra $\mathcal{A}$ is endowed with a generalized $H$-action if there exists a (unital) algebra homomorphism $H\to\mathrm{End}_\mathbb{F}(\mathcal{A})$ such that for each $h\in H$, there exist $h_i'$, $h_i''$, $h_i'''$, $h_i''''\in H$ satisfying
$$
h(ab)=\sum_ih_i'(a)h_i''(b)+h_i'''(b)h_i''''(a),\quad\forall a,b\in\mathcal{A}.
$$
Examples of $H$-action include gradings by groups, involution and actions of Hopf algebras (see, for instance, \cite{Berele}).

An \emph{$H$-ideal} of $\mathcal{A}$ is an ideal invariant under the action of $H$. We say that $\mathcal{A}$ is \emph{$H$-simple} if $\mathcal{A}^2\ne0$ and it contains no non-trivial $H$-ideals. We say that $\mathcal{A}$ is $H$-prime if the product of two nonzero $H$-ideals is again nonzero. Clearly each $H$-simple algebra is $H$-prime. An $H$-algebra $\mathcal{A}$ is said to be \emph{central} if it is unital and $Z(\mathcal{A})=\mathbb{F}$.

We proceed to provide a construction for the free $H$-algebra. Let $X=\{x_1,x_2,\ldots,\}$ and consider the free associative algebra, freely generated by $\{x_i^h\mid h\in H,i\in\mathbb{N}\}$. We denote by $\mathbb{F}_H\langle X\rangle$ the quotient by all the relations of the kind $x_i^{\lambda h_1+h_2}-\lambda x_i^{h_1}-x_i^{h_2}$. Then, given an algebra $\mathcal{A}$ endowed with a generalized $H$-action and a map $X\to\mathcal{A}$, there exists a unique algebra homomorphism $f:\mathbb{F}_H\langle X\rangle\to\mathcal{A}$ such that $f(x_i^h)=h\cdot f(x_i)$, $\forall h\in H$ and $i\in\mathbb{N}$. We denote
$$
\mathrm{Id}_H(\mathcal{A})=\bigcap_{\psi\in\mathrm{Hom}(\mathbb{F}_H\langle X\rangle,\mathcal{A})}\mathrm{Ker}\,\psi,
$$
the set of all $H$-polynomial identities of $\mathcal{A}$.

For each $n\in\mathbb{N}$, let
$$
P_n^H=\mathrm{Span}\{x_{\sigma(1)}^{h_1}\cdots x_{\sigma(n)}^{h_n}\mid h_1,\ldots,h_n\in H,\sigma\in\mathcal{S}_n\},
$$
the space of all multilinear $H$-polynomials. Then, we set
$$
c_n^H(\mathcal{A})=\dim P_n^H/P_n^H\cap\mathrm{Id}_H(\mathcal{A}),\quad n\in\mathbb{N}.
$$
The \emph{$H$-exponent} is defined by
$$
\mathrm{exp}^H(\mathcal{A})=\lim_{n\to\infty}\sqrt[n]{c_n^H(\mathcal{A})},
$$
whenever the limit exists. The existence of $H$-exponent is proved by Gordienko (see \cite{Gordb}), in the following context. Let $\mathcal{A}$ be a finite-dimensional non-nilpotent algebra with a generalized $H$-action over an algebraically closed field of characteristic zero. Assume that $J=J(\mathcal{A})$ is $H$-invariant and $\mathcal{A}=\mathcal{B}+J(\mathcal{A})$, where $\mathcal{B}=\mathcal{B}_1\oplus\cdots\oplus\mathcal{B}_q$ is the direct sum of $H$-simple algebras. Then, \cite[Theorem 5]{Gordb} tells us that $\mathrm{exp}^H(\mathcal{A})$ exists and it coincides with
\begin{equation}\label{gord}
\max\left\{\dim(\mathcal{B}_{i_1}\oplus\cdots\oplus\mathcal{B}_{i_r})\mid\mathcal{B}_{i_1}J\cdots J\mathcal{B}_{i_r}\ne0\right\}.
\end{equation}

\subsection{Some lemmas}
In this section, we aim to prove that the tensor product of a central $H$-simple algebra with a unital $K$-simple algebra is $H\otimes K$-simple. This is a generalization of the classical theory, and we shall base our theory on the book \cite{Bresar}.

We denote by $\hat{H}$ the unitization of $H$, i.e., $\hat{H}=H$ if $H$ is unital, and $\hat{H}=\mathbb{F}1\oplus H$ if $H$ is non-unital, where $1$ acts as unity. It is clear that each algebra with generalized $H$-action has a natural generalized $\hat{H}$-action. Moreover, an ideal is $H$-ideal if and only if it is $\hat{H}$-ideal. In addition, it is clear that several properties are preserved, for instance, $H$-simplicity is equivalent to $\hat{H}$-simplicity; $H$-primality is equivalent to $\hat{H}$-primality, and so on.

\begin{Lemma}\label{tensorproduct}
Let $\mathcal{A}$, $\mathcal{B}$, $H$, $K$ be associative algebras, and assume that $\mathcal{A}$ has a generalized $H$-action and $\mathcal{B}$ has a generalized $K$-action. Then, $\mathcal{A}\otimes_\mathbb{F}\mathcal{B}$ has a generalized $\hat{H}\otimes_\mathbb{F}\hat{K}$-action.
\end{Lemma}
\begin{proof}
Note that we have an algebra homomorphism given by the composition of the tensor product of two homomorphisms and the natural embedding:
$$
\hat{H}\otimes_\mathbb{F}\hat{K}\to\mathrm{End}_\mathbb{F}(\mathcal{A})\otimes_\mathbb{F}\mathrm{End}_\mathbb{F}(\mathcal{B})\to\mathrm{End}_\mathbb{F}(\mathcal{A}\otimes_\mathbb{F}\mathcal{B}).
$$
It is elementary to check that this composition satisfies the conditions to obtain a generalized $\hat{H}\otimes\hat{K}$-action on $\mathcal{A}\otimes\mathcal{B}$.
\end{proof}
In particular, $\mathcal{A}\otimes\mathcal{B}$ has a generalized $H\otimes K$-action, and $H$ and $K$ also acts on $\mathcal{A}\otimes\mathcal{B}$ via $H\otimes\mathrm{Id}_\mathcal{B}$ and $\mathrm{Id}_\mathcal{A}\otimes K$, respectively.

\begin{Lemma}\label{scalar}
Let $\mathcal{A}$ be a finite-dimensional $H$-simple over an algebraically closed field $\mathbb{F}$ and $f:\mathcal{A}\to\mathcal{A}$ an $H$-linear homomorphism of $\mathcal{A}$-bimodules. Then, there exists $\lambda\in\mathbb{F}$ such that $f(x)=\lambda x$, for each $x\in\mathcal{A}$.
\end{Lemma}
\begin{proof}
Let $\mathcal{E}\subseteq\mathrm{End}_\mathbb{F}(\mathcal{A})$ be the subalgebra generated by all left and right multiplication by elements of $\mathcal{A}$ and by the image of $H$. Then, $\mathcal{A}$ is a simple $\mathcal{E}$-module and $f$ is an $\mathcal{E}$-endomorphism. Since $\mathcal{A}$ is finite-dimensional and $\mathbb{F}$ is algebraically closed, Schur's Lemma gives $f\in\mathbb{F}$.
\end{proof}

\begin{Lemma}\label{case1}
Let $\mathcal{A}$ be a finite-dimensional $H$-simple algebra over an algebraically closed field $\mathbb{F}$. Let $a$, $b\in\mathcal{A}$ be such that $(h\cdot a)xb=(h\cdot b)xa$, for all $x\in\mathcal{A}$ and $h\in H\cup\{\mathrm{Id}_\mathcal{A}\}$. Then, $\{a,b\}$ is linearly dependent.
\end{Lemma}
\begin{proof}
Assume that $a\ne0$ and consider the map $f:\mathcal{A}\to\mathcal{A}$ given by
$$
\sum x_i(h_i\cdot a)y_i\mapsto\sum x_i(h_i\cdot b)y_i,\quad x_i,y_i\in\mathcal{A}.
$$
The map is well-defined. Indeed, assume that $\sum x_i(h_i\cdot a)y_i=0$. Then, for each $x\in\mathcal{A}$, one has
$$
0=\sum x_i(h_i\cdot a)y_ixb=\sum x_i(h_i\cdot b)y_i xa=\left(\sum x_i(h_i\cdot b)y_i\right)xa.
$$
Since $\mathcal{A}$ is $H$-prime, one gets $\sum x_i(h_i\cdot b)y_i=0$. Now, it is clear that $f$ is an $H$-linear bimodule homomorphism. Thus, $f(x)=\lambda x$, $\forall x\in\mathcal{A}$, for some $\lambda\in\mathbb{F}$ (\Cref{scalar}). In particular, for any $x$, $y\in\mathcal{A}$,
$$
xby=f(xay)=\lambda xay.
$$
This gives $b=\lambda a$.
\end{proof}

The following is not necessary for our purposes; however, it is interesting by its own and is a direct generalization of the classical situation to the case of generalized $H$-actions.
\begin{proposition}\label{li}
Let $\mathcal{A}$ be a finite-dimensional $H$-simple algebra over an algebraically closed field $\mathbb{F}$ and $\{a_1,\ldots,a_m\}\subseteq\mathcal{A}$ an $\mathbb{F}$-linearly independent set. Let $b_1$, \dots, $b_m\in\mathcal{A}$ be such that
$$
\sum_{i=1}^ma_ix(h\cdot b_i)=0,\quad\forall x\in\mathcal{A}, h\in H\cup\{\mathrm{Id}_\mathcal{A}\}.
$$
Then, $b_1=\cdots=b_m=0$.
\end{proposition}
\begin{proof}
The proof will be by induction on $m$. If $m=1$, then the result follows since $\mathcal{A}$ is $H$-prime. So, let $m>1$ and assume that $b_m\ne0$. Then,
$$
\sum_{i=1}^ma_i(x(h\cdot b_m)y)b_i=0,\quad\forall x, y\in\mathcal{A},\forall h\in H\cup\{\mathrm{Id}_\mathcal{A}\}.
$$
However, $a_mx(h\cdot b_m)=-\sum_{i=1}^{m-1}a_ix(h\cdot b_i)$, for each $x\in\mathcal{A}$. Thus,
$$
\sum_{i=1}^{m-1}a_ix((h\cdot b_m)yb_i-(h\cdot b_i)yb_m),\quad\forall x\in I,y\in I'.
$$
By induction, $(h\cdot b_m)yb_i-(h\cdot b_i)yb_m=0$, $\forall y\in\mathcal{A}$, for each $i=1,2,\ldots,m-1$. From \Cref{case1}, $b_i=\lambda_ib_m$, for some $\lambda_i\in\mathbb{F}$, for each $i=1,\ldots,m-1$. Set $\lambda_m=1$. Hence, one gets $\sum_{i=1}^m(\lambda_ia_i)xb_m=0$, for each $x\in\mathcal{A}$. Since $b_m\ne0$ and $\mathcal{A}$ is $H$-prime, one gets $\sum_{i=1}^m\lambda_ia_i=0$, a contradiction.
\end{proof}

\begin{Thm}\label{tensorsimple}
Let $\mathcal{A}$ be a finite-dimensional $H$-simple algebra and $\mathcal{S}$ a central $K$-simple algebra, both over an algebraically closed field $\mathbb{F}$. Then, $\mathcal{A}\otimes_\mathbb{F}\mathcal{S}$ is an $H\otimes_\mathbb{F}K$-simple algebra.
\end{Thm}
\begin{proof}
Let $I\subseteq\mathcal{A}\otimes\mathcal{S}$ be a nonzero $H\otimes K$-ideal. Let $0\ne x=\sum_{i=1}^ka_i\otimes s_i\in I$ be such that $k$ is minimal. We can assume that $\{s_1,\ldots,s_k\}$ is $\mathbb{F}$-linearly independent. Then,
\begin{align*}
&I\ni (a_kr\otimes1)((h\otimes\mathrm{Id}_\mathcal{S})\cdot x)- x(r(h\cdot a_k)\otimes1)\\&=\sum_{i=1}^{k-1}(a_kr(h\cdot a_i)-a_ir(h\cdot a_k))\otimes s_i,\quad\forall r\in\mathcal{A},\forall h\in\mathcal{H}.
\end{align*}
From the minimality of $k$, such element must be $0$. Since $\{s_1,\ldots,s_{k-1}\}$ is linearly independent, $a_ir(h\cdot a_k)=a_kr(h\cdot a_i)$, $\forall r\in\mathcal{A}$, $\forall h\in H$, and for each $i=1,\ldots,k-1$. From \Cref{case1}, it follows that $a_i=\lambda_ia_k$, for some $\lambda_i\in\mathbb{F}$. Hence, we get $k=1$. Thus, there exists a nonzero element of the kind $a_0\otimes s_0\in I$.

Now, we shall prove that $I=\mathcal{A}\otimes\mathcal{S}$. Since $\mathcal{A}$ is $H$-simple, one has $\hat{\mathcal{A}}(H\cdot a)\hat{\mathcal{A}}=\mathcal{A}$, where $\hat{\mathcal{A}}$ is the unitization of $\mathcal{A}$. Hence,
\begin{equation}\label{temp_eq}
\forall a\in\mathcal{A},\quad a\otimes s_0\in(\hat{\mathcal{A}}\otimes1)\left((H\otimes\mathrm{Id}_\mathcal{S})\cdot(a_0\otimes s_0)\right)(\hat{\mathcal{A}}\otimes1)\subseteq I.
\end{equation}
Now, since $\mathcal{S}$ is unital and $K$-simple, we can find $u_i$, $v_i\in\mathcal{S}$ and $k_i\in K$ such that $\sum_{i=1}^nu_i(k_i\cdot s_1)v_i=1$. Since $\mathcal{A}$ is $H$-simple, $\mathcal{A}a_0\mathcal{A}\ne0$. Thus, we can find $x$, $y\in\mathcal{A}$ such that $xa_0y\ne0$. Finally,
\begin{align*}
I\ni&\sum_{i=1}^n(x\otimes u_i)\left((\mathrm{Id}_\mathcal{A}\otimes k_i)\cdot(a_0\otimes s_0)\right)(y\otimes v_i)\\&=xa_0y\otimes\left(\sum_{i=1}^nu_i(k_i\cdot s_0)v_i\right)=xa_0y\otimes1.
\end{align*}
From \eqref{temp_eq}, we get $a\otimes1\in I$, for all $a\in\mathcal{A}$. This gives $I=\mathcal{A}\otimes_\mathbb{F}\mathcal{S}$. The proof is complete.
\end{proof}

\subsection{Generalized main result}
We have all the steps to state a (partial) generalization of our main result.
\begin{Thm}\label{generalization_mainthm}
We let $H$ and $K$ be unital associative algebras, let $\mathcal{A}$ be a finite-dimensional associative algebra with generalized $H$-action and $\mathcal{S}$ be a unital finite-dimensional associative algebra with generalized $K$-action, where $\mathcal{S}$ is central and $K$-simple, all over an algebraically closed field $\mathbb{F}$ of characteristic zero. Assume that $J(\mathcal{A})$ is $H$-invariant and $\mathcal{A}$ is the sum of its $H$-semisimple part and $J(\mathcal{A})$. Then,
$$
\mathrm{exp}^{H\otimes_\mathbb{F}K}(\mathcal{A}\otimes_\mathbb{F}\mathcal{S})=\dim\mathcal{S}\cdot\mathrm{exp}^H(\mathcal{A}).
$$
\end{Thm}
\begin{proof}
One has $\mathcal{A}=(\mathcal{B}_1\oplus\cdots\oplus\mathcal{B}_m)+J(\mathcal{A})$, where each $\mathcal{B}_i$ is $H$-simple and finite-dimensional. From \Cref{tensorproduct}, $\mathcal{A}\otimes\mathcal{S}$ has an $H\otimes K$-action. Moreover, one has
$$
\mathcal{A}\otimes\mathcal{S}=(\mathcal{B}_1\otimes\mathcal{S}\oplus\cdots\oplus\mathcal{B}_m\otimes\mathcal{S})+J(\mathcal{A})\otimes\mathcal{S}.
$$
It is clear that each $\mathcal{B}_i\otimes\mathcal{S}$ is an $H\otimes K$-subalgebra, and it is $H\otimes K$-simple (\Cref{tensorsimple}). In addition, $J(\mathcal{A})\otimes\mathcal{S}$ is an $H\otimes K$-ideal and it is nilpotent. Thus, the proof follows the same steps as in the proof of \Cref{mainthm}, i.e., $\mathcal{B}_{i_1}J(\mathcal{A})\cdots J(\mathcal{A})\mathcal{B}_{i_t}\ne0$ if and only if $(\mathcal{B}_{i_1}\otimes\mathcal{S})(J(\mathcal{A})\otimes\mathcal{S})\cdots(J(\mathcal{A})\otimes\mathcal{S})(\mathcal{B}_{i_t}\otimes\mathcal{S})\ne0$. Then, we use the formula \eqref{gord}.
\end{proof}

As a particular cases of the previous theorem, we have:
\begin{Cor}
Let $\mathcal{A}$ be a finite-dimensional associative algebra over an algebraically closed field of characteristic zero.
\begin{enumerate}
\renewcommand{\labelenumi}{(\arabic{enumi})}
\item Let $G$ be an abelian group and consider a finite-dimensional unital $G$-graded graded-simple algebra $\mathcal{S}$ and a $G$-grading on $\mathcal{A}$. Then, $\mathcal{A}\otimes_\mathbb{F}\mathcal{S}$ is $G$-graded via $\left(\mathcal{A}\otimes\mathcal{S}\right)_g=\sum_{g_1g_2=g}\mathcal{A}_{g_1}\otimes\mathcal{S}_{g_2}$, and
$$
\mathrm{exp}^G(\mathcal{A}\otimes_\mathbb{F}\mathcal{S})=\dim\mathcal{S}\cdot\mathrm{exp}^G(\mathcal{A}).
$$
\item Let $G_1$ and $G_2$ be (not necessarily abelian) groups and consider a $G_1$-grading on $\mathcal{A}$ and a finite-dimensional unital $G_2$-graded graded-simple algebra $\mathcal{S}$. Then, $\mathcal{A}\otimes_\mathbb{F}\mathcal{S}$ is $G_1\times G_2$-graded via $\left(\mathcal{A}\otimes\mathcal{S}\right)_{(g_1,g_2)}=\mathcal{A}_{g_1}\otimes\mathcal{S}_{g_2}$, and
$$
\mathrm{exp}^{G_1\times G_2}(\mathcal{A}\otimes_\mathbb{F}\mathcal{S})=\dim\mathcal{S}\cdot\mathrm{exp}^{G_1}(\mathcal{A}).
$$
\item Let $\mathcal{S}$ be a finite-dimensional unital $\ast$-simple associative algebra with involution and $\mathcal{A}$ is endowed with an involution such that $\mathcal{A}=\mathcal{B}+J(\mathcal{A})$, with $\mathcal{B}$ being $\ast$-semisimple. Then, $\mathcal{A}\otimes_\mathbb{F}\mathcal{S}$ has an involution via $(a\otimes s)^\ast=a^\ast\otimes s^\ast$, and
$$
\mathrm{exp}^\ast(\mathcal{A}\otimes_\mathbb{F}\mathcal{S})=\dim\mathcal{S}\cdot\mathrm{exp}^\ast(\mathcal{A}).
$$
\item Let $H$ be a Hopf algebra such that the square of its antipode is the identity, assume that $\mathcal{A}$ is endowed with an $H$-action and $\mathcal{S}$ is a finite-dimensional unital associative algebra endowed with an $H$-action where it is $H$-simple. Then, $\mathcal{A}\otimes_\mathbb{F}\mathcal{S}$ has an $H$-action and
$$
\mathrm{exp}^H(\mathcal{A}\otimes_\mathbb{F}\mathcal{S})=\dim\mathcal{S}\cdot\mathrm{exp}^H(\mathcal{A}).
$$
\end{enumerate}
\end{Cor}
\begin{proof}
In order to apply \Cref{generalization_mainthm}, it is enough to show that $J(\mathcal{A})$ is invariant under the extra structures. The cases (1) and (2) follow from \cite[Corollary 3.3]{Gord} (also follows from \cite[Theorem 2.1]{KO}). The situation (3) is immediate. The case (4) is \cite[Theorem 3.2]{Gord}.
\end{proof}

For the particular case of matrix algebras with coefficients in an algebra, we have:
\begin{Cor}
Let $\mathcal{A}$ be a finite-dimensional associative algebra over an algebraically closed field of characteristic zero and $n\in\mathbb{N}$.
\begin{enumerate}
\item Assume that $G$ is an abelian group and $\mathcal{A}$ and $M_n(\mathbb{F})$ are endowed with $G$-gradings. Then, $M_n(\mathcal{A})$ is $G$-graded and
$$
\mathrm{exp}^G(M_n(\mathcal{A}))=n^2\mathrm{exp}^G(\mathcal{A}).
$$
\item Assume that $\mathcal{A}$ and $M_n(\mathbb{F})$ are endowed with involutions such that $\mathcal{A}=\mathcal{B}+J(\mathcal{A})$, with $\mathcal{B}$ being $\ast$-semisimple. Then, $M_n(\mathcal{A})$ has an involution and
$$
\mathrm{exp}^\ast(M_n(\mathcal{A}))=n^2\mathrm{exp}^\ast(\mathcal{A}).
$$
\item Let $H$ be a Hopf algebra such that the square of its antipode is the identity and assume that $\mathcal{A}$ and $M_n(\mathbb{F})$ are endowed with $H$-actions. Then, $M_n(\mathcal{A})$ has an $H$-action and
$$
\mathrm{exp}^H(M_n(\mathcal{A}))=n^2\mathrm{exp}^H(\mathcal{A}).
$$
\end{enumerate}
\end{Cor}
\begin{proof}
Note that $\mathcal{S}=M_n(\mathbb{F})$ is a central simple associative algebra. Thus, it will remain central and simple when endowed with an extra structure. The result is a specialization of the previous corollary.
\end{proof}

\section*{Acknowledgments}
We thank Dr.~Javier S\'anchez Serd\`a for pointing out Remark 1.5 of \cite{Lam}.


\begin{thebibliography}{XX}
\bibitem{AGPR} E.~Aljadeff, A.~Giambruno, C.~Procesi, A.~Regev, \emph{Rings with polynomial identities and finite dimensional representations of algebras}. American Mathematical Society Colloquium Publications 66. American Mathematical Society, Providence, RI, 2020.
\bibitem{Berele} A.~Berele, \emph{Cocharacter sequences for algebras with Hopf algebra actions}, Journal of Algebra \textbf{185} (1996), 869--885.
\bibitem{Bresar} M.~Bre\v sar, \emph{Introduction to noncommutative algebra}. Universitext. Springer, Cham, 2014.
\bibitem{deMello} T.~de Mello, \emph{If $M_n(R)$ and $M_m(R)$ satisfy the same polynomial identities is it true that $m=n$?},  \url{https://mathoverflow.net/q/116698} URL (version: 2013-05-13).
\bibitem{GZ1999} A.~Giambruno, M.~Zaicev, \emph{Exponential codimension growth of PI algebras: an exact estimate}, Advances in Mathematics \textbf{142} (1999), 221--243.
\bibitem{GZ2005} A.~Giambruno, M.~Zaicev, \emph{Polynomial identities and asymptotic methods}. American Mathematical Society, Providence, RI, 2005.
\bibitem{Gordb} A.~Gordienko, \emph{Amitsur’s conjecture for associative algebras with a generalized Hopf action}, Journal of Pure and Applied Algebra \textbf{217} (2013), 1395--1411.
\bibitem{Gord} A.~Gordienko, \emph{Co-stability of radicals and its applications to PI-theory}, Algebra Colloquium \textbf{23} (2016), 481--492.
\bibitem{KO} A.~Kelarev, J.~Okni\'nski, \emph{The Jacobson radical of graded PI-rings and related classes of rings}, Journal of Algebra \textbf{186} (1996), 818--830.
\bibitem{kemerbook} A.~Kemer, \emph{Ideals of Identities of Associative Algebras}. Translations of Mathematical Monographs, vol. 87, American Mathematical Society, 1991.
\bibitem{Lam} T.~Y.~Lam, \emph{Lectures on Modules and Rings}. Graduate Texts in Mathematics,  Springer New York, NY, 2012.
\end{thebibliography}
\end{document}